%% file: conference_101719.tex
\documentclass[letterpaper, 10 pt, conference]{ieeeconf}  

\IEEEoverridecommandlockouts                              

\title{\LARGE \bf
Commutator-Driven Stability Bounds for Periodic Switching }

\author{
Debanjan Mallik$^{1}$,
Nikhil Chopra$^{1}$
\thanks{$^{1}$The authors are with Institute for Systems Research, 
        University of Maryland, College Park, MD 20742, USA.
        {(e-mail: \tt\small dmallik@umd.edu; nchopra@umd.edu}).}
\thanks{This article has been accepted for publication in IEEE Control Systems Letters.
        DOI: 10.1109/LCSYS.2026.3708744.
        \copyright~2026 IEEE. Personal use of this material is permitted.
        Permission from IEEE must be obtained for all other uses, in any current or
        future media, including reprinting/republishing this material for advertising
        or promotional purposes, creating new collective works, for resale or
        redistribution to servers or lists, or reuse of any copyrighted component of
        this work in other works.}
}

\usepackage{cite}
\usepackage{amsmath,amssymb,amsfonts}

\usepackage{amsthm}
\usepackage{pgfplots}
\usepgfplotslibrary{fillbetween}
\pgfplotsset{compat=1.18}
\newtheorem{theorem}{Theorem} 
\newtheorem{lemma}[theorem]{Lemma} 

\newtheorem{remark}[theorem]{Remark}
\usepackage{graphicx}
\usepackage{textcomp}
\usepackage{xcolor}
\usepackage{verbatim}
\usepackage{algorithm}
\usepackage{algorithmic}

\def\BibTeX{{\rm B\kern-.05em{\sc i\kern-.025em b}\kern-.08em
    T\kern-.1667em\lower.7ex\hbox{E}\kern-.125emX}}

\begin{document}
\maketitle

\begin{abstract}
Averaged models are widely used to analyze periodically switched linear systems, yet the stability of the averaged flow does not automatically guarantee the stability of the true switched dynamics. The discrepancy arises from noncommutativity among the subsystem generators,
so stability certificates benefit from bounds that expose this dependence
in a form compatible with Lyapunov contraction metrics. We derive an explicit operator-norm bound for the one-period mismatch between the switched and averaged propagators, in which the leading-order error depends explicitly on the pairwise commutator norms of the scaled mode generators, with a closed-form prefactor depending only on the generator norms.
This bound yields a computable threshold for the switching period below which the switched system inherits exponential stability from its averaged model, uniformly certified over admissible duty fractions. The analysis extends to an arbitrary number of switching modes via telescoping induction, and a semidefinite program provides sampled duty-dependent Lyapunov metrics for implementing the certificate.

\end{abstract}

\input{intro}
\input{twomode}
\input{mmode}
\input{example}

\section{Conclusion and Future Work}
This letter derived explicit commutator-norm bounds for the mismatch between cyclic switched propagators and their averaged flows. Combined with Lyapunov contraction of the averaged generator, these bounds give computable switching-period certificates for fixed cyclic \(m\)-mode schedules. A three-mode numerical example illustrated the sampled computation and showed how commutator-graph edges can affect the sufficient certificate even when the corresponding interaction is invisible to the Floquet spectrum. Future work will pursue sharper bounds, richer metric parameterizations, and extensions beyond fixed cyclic orders.
\bibliographystyle{IEEEtran}
\bibliography{ref}

\end{document}

%% file: intro.tex
\section{Introduction}
\label{intro}

Periodically switched linear systems admit two complementary descriptions: the one-period state-transition (monodromy) matrix, given by a product of exponentials, and the averaged flow, determined by convex combinations of the subsystem matrices. We consider the switched linear system
\[
    \dot x(t)=A_{\sigma(t)}x(t),\qquad x(t)\in\mathbb R^n,
\]
under periodic switching. In this letter, the switching signal is
restricted to a fixed cyclic schedule: over each period \(\tau\), the
modes are visited in the prescribed order \(1,\ldots,m\), and mode
\(i\) is active for time \(\alpha_i\tau\), where
\[
    \alpha\in\Delta^{m-1}:=
    \{\alpha\in\mathbb R^m_{\ge0}:\sum_{i=1}^m\alpha_i=1\}.
\] Thus, cyclic switching is the structured periodic subclass considered
here. For fast switching, the averaged model often captures the dominant stability behavior, but the inheritance of stability by the true system is obstructed by noncommutativity among the modes. This motivates the search for explicit bounds that quantify the conditions under which averaged contraction dominates the splitting error. In this work, we derive such bounds and obtain uniform exponential stability certificates for admissible duty-cycle sets.

Stability of periodically switched linear systems has long been studied
through direct analysis of finite products of exponentials~\cite{tokarzewski1987}
and Floquet-theoretic characterizations of the monodromy~\cite{gokcek2004}.
Related existence results show that if a stable convex combination of
subsystem matrices exists, then a periodic switching law can be constructed
to asymptotically stabilize the overall system~\cite{xie2005}. In the
fast-switching regime, Porfiri, Roberson and Stilwell~\cite{porfiri2008}
combine classical averaging methods~\cite{ezzine1989} with exponential
splitting to derive computable switching-period bounds that expose the
role of commutators. Averaging-based stability results
without commutativity assumptions have also been developed in related
settings, including switched-DAE formulations~\cite{MironchenkoWirthWulff2015,
MostacciuoloTrennVasca2017}. For a broad survey of
stability criteria for switched and hybrid systems, including common
Lyapunov functions, dwell-time constraints, and periodic versus arbitrary
switching, see~\cite{shorten2007}.

Lie-algebraic stability criteria show that commutation relations among subsystem matrices can fundamentally shape switched system stability under arbitrary switching, with solvability-type conditions yielding common quadratic Lyapunov functions and uniform exponential stability~\cite{liberzon1999, agrachev2001}; for a recent retrospective on this commutator-relations viewpoint, see~\cite{liberzon2023}. Robust Lie-algebraic criteria for arbitrary switching were developed in~\cite{agrachev2012}; in continuous time, these criteria use Levi and Cartan decompositions to quantify closeness to solvable or solvable-plus-compact Lie-algebraic structure. Commutator conditions have also been used to characterize stabilizing switching signals under dwell-time restrictions~\cite{kundu2019,kundu2020}, where bounds on commutator norms serve as qualitative existence conditions
for stable switching sequences. The present work is complementary to these lines of work; it uses commutator norms quantitatively, embedding them
directly into a computable period threshold for fixed
cyclic schedules.

The role of commutators in exponential splitting is classical from the BCH expansion and from Lie-algebraic approaches. The contribution here is not the identification of \([X,Y]\) as the
leading obstruction, but a duty-uniform and order-dependent Lyapunov
certificate built from a self-contained induced-norm splitting estimate with a computable closed-form prefactor. More specifically:
\begin{itemize}
\item We obtain a two-factor estimate for
\(\|e^Xe^Y-e^{X+Y}\|_Q\) with a closed-form prefactor
depending on \(\|X\|_Q\) and \(\|Y\|_Q\), multiplicative dependence on
\(\|[X,Y]\|_Q\), and no hidden compactness constant. In contrast to the matrix measure growth factors in~\cite{porfiri2008}, the growth factors here are expressed through ordinary induced-norm quantities in the same \(Q\)-metric used to state the splitting bound, making the bound directly compatible with the Lyapunov metric used in the certificate.
\item We extend the estimate to \(m\)-mode cyclic products by a
telescoping argument and obtain a duty- and order-dependent sufficient
condition \(\kappa^\ast(\tau)<1\) for uniform exponential stability over
\(\mathcal A\subseteq\Delta^{m-1}\).
\item We give a sampled SDP implementation using affine
duty-dependent metrics \(P(\alpha)=\sum_i\alpha_iP_i\), which permits
the Lyapunov metric to vary with the duty vector. A three-mode example illustrates the sampled computation, the commutator-graph interpretation in Remark~\ref{remark:commutator_graph}, and comparison with a
sampled Floquet threshold. \end{itemize}

%% file: twomode.tex
\section{Two-Mode Lyapunov Contraction}
\label{lyap_cont}
Let \(\mathcal{V} = \mathbb{R}^n\) and $G = \mathrm{GL}(\mathcal{V})\cong \mathrm{GL}(n, \mathbb{R})$ with Lie algebra $\mathfrak{g} = \mathfrak{gl}(\mathcal{V}).$ For symmetric matrices, \(X\preceq Y\) means that \(Y-X\) is positive semidefinite, and \(\mathbb S^n_{++}\) denotes the cone of real symmetric positive definite \(n\times n\) matrices. Fix \(P \in \mathbb{S}_{++}^n\) and equip \(\mathcal{V}\) with the inner product in the $P$-metric \(\langle x,y\rangle_P := x^\top P y\) and norm \(\|x\|_P = \sqrt{x^\top Px}\). The induced operator norm is \[\|M\|_P = \sup_{x\neq 0} \frac{\|Mx\|_P}{\|x\|_P}= \sup_{\|x\|_P=1}\|Mx\|_P, \quad M \in \mathrm{End}(\mathcal{V}).\] Let the quadratic Lyapunov function be \(W(x) = \|x\|_P^2 = x^\top Px.\) Moreover, let \(A_1, A_2 \in \mathfrak{g}\). For \(\alpha \in (0, 1)\), define \(B_1(\alpha):= \alpha A_1\), \(B_2(\alpha) := (1-\alpha) A_2\), and \( B(\alpha) := B_1(\alpha) + B_2(\alpha) :=  \alpha A_1+(1-\alpha) A_2 = \bar A(\alpha)\). For a period \(\tau >0\), the monodromy matrix or one-period propagator is \[\Phi(\tau, \alpha) := \mathrm{exp}\left(\tau B_2\left(\alpha\right)\right) \mathrm{exp}(\tau B_1(\alpha)) \in \,G\,\] 

\noindent and the averaged propagator is \(\Psi(\tau, \alpha) := \mathrm{exp}(\tau B(\alpha)).\)
\begin{lemma} [Contraction] 
     Assume there exist a closed interval $I \subset (0,1)$ and constants $P\succ 0$, $\eta>0$ such that the LMI \begin{equation}B(\alpha)^\top P + PB(\alpha) \preceq -2\eta P,\quad\alpha \in I
    \label{eq:1}
    \end{equation} holds. Then for all \(\alpha \in I\) and all \(t\ge0\),
    \[\|\Psi(t, \alpha)\|_P\ = \| \mathrm{exp}(tB(\alpha))\|_P \le e^{-\eta t}.\] 
    \label{lem:contraction in the uniformly continuous semigroup}
\end{lemma}
\begin{proof}
 For the LTI system $\dot{x}(t) = B(\alpha)x(t)$ with $x(0) = x_0$, the solution is
$x(t) = e^{tB(\alpha)}x_0$. Differentiating $W = x^\top P x$ along trajectories gives
\[
    \dot{W} = x^\top\!\left(B(\alpha)^\top P + PB(\alpha)\right)x \le -2\eta W.
\]
Hence $W(t) \le e^{-2\eta t}W(0)$, and therefore
\[
    \|e^{tB(\alpha)}x_0\|_P \le e^{-\eta t}\|x_0\|_P.
\]
Taking the supremum over $\|x_0\|_P = 1$ proves the lemma.
\end{proof}
For fixed \(P\), condition~\eqref{eq:1} is a common quadratic Lyapunov inequality for the averaged family \(B(\alpha)\), \(\alpha\in I\), not for the individual modes unless endpoint duties are included. Thus \(B(\alpha)\) is Hurwitz on \(I\), while \(A_1\) and \(A_2\) need not be. Since the LHS of \eqref{eq:1} is affine in \(\alpha\), checking the endpoints \(\alpha_-\) and \(\alpha_+\) of a closed
interval \(I=[\alpha_-,\alpha_+]\) suffices. This shortcut is special to the fixed-metric two-mode case. The restriction \(I\subset(0,1)\) avoids
zero-dwell degeneracies; endpoints can be included by continuity when
allowed. With duty-dependent metrics \(P(\alpha)\) used later, the
Lyapunov inequality is generally nonlinear in \(\alpha\), so we enforce
it on a finite duty grid.

\subsection{Splitting error and the commutator ideal}
For the set of matrices \(\{A_1, A_2\}\), we define the smallest Lie algebra generated by $A_1$ and $A_2$ as
$\mathfrak{h} :=  \{A_1, A_2\}_{\mathrm{Lie}} \subset \mathfrak{g}.$ Since $\mathfrak{h}$ is closed under scalar multiplication and $\alpha\in(0,1)$, $\{ B_1(\alpha), B_2(\alpha)\}_{\mathrm{Lie}} = \mathfrak{h}.$
We know the map \((\tau, \alpha) \mapsto \Phi(\tau, \alpha)\) is analytic, and \(\Phi (0, \alpha) = \mathbb{I}_n\). Since the principal matrix logarithm is well-defined on an open neighborhood \(U \subset \mathrm{GL}(n)\) of $\mathbb{I}_n$~\cite{hall2015}, there exists \(\varepsilon>0\) such that the norm ball \(\mathcal{B}_\varepsilon(\mathbb{I}_n) \subset U\). Fix \(\delta_0>0\), and consider the compact set \(K:= [-\delta_0, \delta_0] \times I\). From Heine-Cantor, continuity of \(\Phi\) ensures its uniform continuity over \(K\). Thus, for any chosen \(\varepsilon>0\), there exists \(\rho >0\) such that for all \((\tau, \alpha) \in K\),
\begin{equation}|\tau -0| < \rho \implies \|\Phi(\tau, \alpha) - \mathbb{I}_n\| < \varepsilon, \quad \text{for all } \alpha \in I.
\label{eq:2}
\end{equation}
Setting \(\tau_\mathrm{log}:= \min\{\rho, \delta_0\}\), \eqref{eq:2} holds for all \(\alpha\in I\) and all \(|\tau|<\tau_\mathrm{log}\), hence, \(\Phi(\tau, \alpha)\in\mathcal{B}_\varepsilon \subset U.\) Therefore, under these conditions, the principal logarithm \(\Omega(\tau, \alpha) := \mathrm{log}\,\Phi(\tau, \alpha)\) is well-defined and analytic.

\par Applying Baker-Campbell-Hausdorff formula gives,
\begin{equation}
    \Omega(\tau, \alpha) = \tau(B_1(\alpha) + B_2(\alpha) ) \,+ \Delta(\tau, \alpha), \quad \Delta\in [\mathfrak{h}, \mathfrak{h}].
    \label{eq:3}
\end{equation}
Here, $[\mathfrak{h}, \mathfrak{h}]$ is the commutator ideal. In particular, the error term in \eqref{eq:3} can be expanded as
\begin{equation*}
    \Delta(\tau, \alpha) 
    = \frac{\tau^2}{2}\alpha(1-\alpha)[A_2, A_1] + O(\tau^3).
\end{equation*}
\begin{lemma} [\(O(\tau^2)\) splitting bound]
Let \(I \subset(0,1)\) be a compact interval. Then, there exist constants \(\tau_0>0, \, C_I\ge 0\) such that \(\text{for all } \alpha\in I, \, \tau\in[0, \tau_0],\) 
\[\|\Phi(\tau, \alpha) - \Psi(\tau, \alpha)\|_P \le C_I\tau^2.\]
\label{lem:uniform_splitting_bound}  
\end{lemma}
\begin{proof}
Let \(F(\tau,\alpha):=\Phi(\tau,\alpha)-\Psi(\tau,\alpha)\). Then
\(F(0,\alpha)=0\) and
\(\partial_\tau F(0,\alpha)=0\).
Since matrix exponentials are analytic, Taylor's formula with remainder gives
\[
F(\tau,\alpha)=\int_0^\tau(\tau-s)\partial_{\tau\tau}F(s,\alpha)\,ds .
\]
By compactness, \(M_I:=\sup_{(s,\alpha)\in[0,\tau_0]\times I}
\|\partial_{\tau\tau}F(s,\alpha)\|_P<\infty\) exists, and hence
\(\|F(\tau,\alpha)\|_P\le (M_I/2)\tau^2\). Set \(C_I := {M_I}/{2}\). This proves the lemma.
\end{proof}

\subsection{One-step contraction for the monodromy map}
\begin{theorem} [$2$-mode Lyapunov contraction]
\label{thm:lyap_splitting_rep}
    There exist \(\tau^\ast \in [0,\tau_0]\) and \(C_I \ge0\) such that for all \(\alpha\in I\) and all \(\tau \in (0,\tau^\ast)\), 
    \[W(\Phi(\tau, \alpha)x) \le q(\tau)W(x);\, q(\tau):= (e^{-\eta\tau}+C_I\tau^2)^2 <1.\]
    Consequently, for any fixed \(\tau \in (0, \tau^\ast)\), the discrete-time system \(x_{k+1} = \Phi(\tau,\alpha)x_k\) is exponentially stable uniformly \(\text{for all } \alpha \in I\), and so is the underlying continuous-time periodic switched system. 
\end{theorem}
\begin{proof}
Triangle inequality, Lemma~\ref{lem:contraction in the uniformly continuous semigroup}
and ~\ref{lem:uniform_splitting_bound} give,
\[
\|\Phi(\tau,\alpha)x\|_P
\le
\big(e^{-\eta\tau}+C_I\tau^2\big)\|x\|_P .
\]
Hence
\[W(\Phi(\tau,\alpha)x)\le q(\tau)W(x),\] with \(q(\tau)\) as stated. Since
\[
e^{-\eta\tau}+C_I\tau^2
\le
1-\eta\tau+\big(\tfrac12\eta^2+C_I\big)\tau^2,
\]
we have \(q(\tau)<1\) whenever
$0<\tau<\min\left\{\tau_0,\frac{\eta}{\frac12\eta^2+C_I}\right\}.$

Iterating the one-step contraction gives
\(
\|x_k\|_P\le q(\tau)^{k/2}\|x_0\|_P
=
e^{-\lambda_d k}\|x_0\|_P,\text{ with }
\lambda_d:=-\tfrac12\log q(\tau)>0 .
\)
The continuous-time conclusion follows from boundedness of the
within-period propagator on the compact set \(s\in[0,\tau]\),
\(\alpha\in I\). This proves the theorem.
\end{proof}

%% file: mmode.tex
\section{$m$-Mode Generalization}
\label{m-mode}
For \(m\) modes in the fixed cyclic order \(1,\ldots,m\), let
\(\alpha\in\mathcal A\subseteq\Delta^{m-1}\) denote the duty vector. Define the scaled generators
\(B_i(\alpha):= \alpha_iA_i\), and the averaged generator, \(S(\alpha) := \sum_{i=1}^m B_i(\alpha) = \sum_{i=1}^m \alpha_i A_i\). The one-period monodromy map is
$\Phi(\tau, \alpha):= e^{\tau B_m(\alpha)}\cdots e^{\tau B_1(\alpha)},$ and the averaged flow map is $\Psi(\tau, \alpha):= e^{\tau S(\alpha)}.$

As the number of switching modes increases, existence of a single common quadratic Lyapunov function (CQLF) becomes a stricter requirement to satisfy. Thus, in the general analysis, we will use parameterized Lyapunov functions, while keeping the theoretical backbone structurally similar to the two-mode case. We take a continuous map \(P: \mathcal{A} \to \mathbb{S}_{++}^n\), and \(\alpha \mapsto P(\alpha)\), defined on a compact subset \(\mathcal{A}\subseteq\Delta^{m-1}\) of the simplex. 
\begin{lemma} [Contraction in the $P(\alpha)$ metric]
\label{lem:avg-mmode}
Let there exist \(\eta>0\) and a map \(P:\mathcal{A}\to \mathbb{S}_{++}^n\) such that for all \(\alpha \in \mathcal{A}\),
\begin{equation*}S(\alpha)^\top P(\alpha) + P(\alpha) S(\alpha) \preceq -2\eta P(\alpha).
\label{eq:lyap}
\end{equation*}
Then $\text{for all } t\ge0$ and \(\alpha \in \mathcal{A},\)
\begin{equation}
    \|e^{tS(\alpha)}\|_{P(\alpha)} \le e^{-\eta t}.
    \label{eq:4}
\end{equation}
\end{lemma}
\begin{proof}
    Fix \(\alpha \in \mathcal{A},\) and consider the averaged system \(\dot x = S(\alpha)\). Let \(V_\alpha(x):= x^\top P(\alpha)x = \|x\|^2_{P(\alpha)}\). Along the trajectories, \(\dot{V_\alpha}= x^\top(S(\alpha)^\top P(\alpha) + P(\alpha) S(\alpha)) x \le -2\eta x^\top P(\alpha) x = -2\eta V_{\alpha}\). Thus, \(V_\alpha(t) \le e^{-2\eta t}V_\alpha(0)\), and \(\|x(t)\|_{P(\alpha)} \le e^{-\eta t}\|x(0)\|_{P(\alpha)}.\) Taking supremum over \(\|x(0)\|_{P(\alpha)} =1\) gives \eqref{eq:4}. This proves the lemma. 
\end{proof}
Define a scalar function, 
\[g(z):= \int_0^1 (1-s) e^{2zs}ds =\begin{cases}
    \frac{e^{2z} -1-2z}{4z^2}, \quad z\ne0, \\
    \frac{1}{2}, \qquad \,\qquad z=0.
\end{cases}\]
\begin{lemma} [Two-factor commutator splitting bound]
\label{lem:twofactor}
    Fix \(Q\succ 0\). For any \(X, Y \in \mathfrak{gl}(n)\), 
    \[\|e^Xe^Y - e^{X+Y}\|_Q \le e^{(\|X\|_Q + \|Y\|_Q)}\,g(\|X\|_Q)\,\|[X, Y]\|_Q.\]
\end{lemma}
\begin{proof}
    For \(s\in[0,1]\), set \(\phi(s) := e^{sX} e^{sY}\), \(\psi(s):= e^{s(X+Y)}\) and \(E(s):= \phi(s) - \psi(s).\) Here, \(E(0) = \phi(0) - \psi(0) = 0\). Differentiating \(\phi(s)\),
    \[\phi'(s) = \frac{d}{ds}(e^{sX}e^{sY}) = Xe^{sX}e^{sY}+ e^{sX}Ye^{sY}. \]
    We know,  
    \[(X+Y) \phi(s) = Xe^{sX}e^{sY}+ Ye^{sX}e^{sY}.\]
    \noindent Thus
    \begin{equation}
        \phi'(s) - (X+Y) \phi(s) = (e^{sX} Ye^{-sX} - Y) \phi(s).
        \label{eq:5}
    \end{equation}
Define $\Lambda(s):=e^{sX}Ye^{-sX}-Y$. Since $\psi'(s)=(X+Y)\psi(s)$,
subtracting from \eqref{eq:5} and writing $E=\phi-\psi$ gives
\begin{equation}
  E'(s) = (X+Y)E(s)+\Lambda(s)\phi(s),\quad E(0)=0. \label{eq:6}
\end{equation}
    From \eqref{eq:6}, variation of constants (with \(s=1\)) yields, 
    \begin{equation}
        E(1) = \int_0^1 e^{(1-s) (X+Y)}\Lambda(s) \phi(s) ds.
        \label{eq:7}
    \end{equation}
    Let \(f(u):= e^{uX}Ye^{-uX}\). Then, \(f'(u) = e^{uX}[X,Y]e^{-uX}\). Note \(f(0) =Y\), and from fundamental theorem of calculus, 
    \begin{equation}
        \Lambda(s) = \int_0^s e^{uX} [X,Y]e^{-uX}du.
        \label{eq:8}
    \end{equation}
    Again, noncommutativity enters only through the commutator.
    Taking the \(Q\)-norm on \eqref{eq:7}, from submultiplicativity,
    \[\|E(1)\|_Q \le \int_0^1 \|e^{(1-s)(X+Y)}\|_Q \,\|\Lambda(s)\|_Q \,\|\phi(s)\|_Q.\]
    Now,
    \[\|e^{(1-s)(X+Y)}\|_Q \le e^{(1-s)\|X+Y\|_Q} \le e^{(1-s)(\|X\|_Q + \|Y\|_Q)},\]
    and \[\|\phi(s)\|_Q \le \|e^{sX}\|_Q \|e^{sY}\|_Q \le e^{s(\|X\|_Q + \|Y\|_Q)}.\]
    Multiplying, 
    \begin{equation}
       \|e^{(1-s)(X+Y)}\|_Q\,\|\phi(s)\|_Q \le e^{(\|X\|_Q + \|Y\|_Q)}. 
       \label{eq:9}
    \end{equation}
  From \eqref{eq:8} and \eqref{eq:9}, 
  \begin{multline}
      \|\Lambda(s)\|_Q\le \int_0^s\|e^{uX}\|_Q\, \|[X,Y]\|_Q\,\|e^{-uX}\|_Q
      \\
      \le \|[X,Y]\|_Q \int_0^s e^{2u\|X\|_Q} \,du.
      \label{eq:10}
  \end{multline}
    From \eqref{eq:10}, Fubini's theorem gives,
    \begin{multline*}
        \|E(1)\|_Q =\|e^Xe^Y - e^{X+Y}\|_Q  \\\le e^{\|X\|_Q+\|Y\|_Q}\|[X,Y]\|_Q\int_0^1\int_0^se^{2u\|X\|_Q}\,du\,ds \\
        = e^{\|X\|_Q+\|Y\|_Q}\|[X,Y]\|_Q\int_0^1\int_u^1e^{2u\|X\|_Q} \,ds \,du\\
        =e^{\|X\|_Q+\|Y\|_Q}\|[X,Y]\|_Q\int_0^1(1-u)\,e^{2u\|X\|_Q}\,du\\
        =e^{\|X\|_Q+\|Y\|_Q}\|[X,Y]\|_Q \frac{e^{2\|X\|_Q} - 1 -2\|X\|_Q}{4\|X\|_Q^2}\\
        =e^{\|X\|_Q+\|Y\|_Q} \,g(\|X\|_Q)\,\|[X,Y]\|_Q.
    \end{multline*}
    This proves the lemma.
    
\end{proof}

Thus the kernel \(g\) is the integral of the conjugation-growth bound over the triangular domain \(0\le u\le s\le1\).

We now apply Lemma~\ref{lem:twofactor} to the \(m\)-factor product \(\Phi\). For fixed \(\alpha \in \mathcal{A}\), define partial sums 
\[S_k(\alpha) := \sum_{i=1}^kB_i(\alpha); \quad k =1, \dots, m, \]
partial products 
\[\Phi_k(\tau, \alpha):= e^{\tau B_k(\alpha)}\cdots e^{\tau B_1(\alpha)},\]
and partial averaged flows 
\[\Psi_k(\tau, \alpha):= e^{\tau S_k(\alpha)}.\]
Define the stage error \(E_k:= \Phi_k -\Psi_k\) with \(E_1 \equiv 0\), and the stage norms with \(1\le i\le k\le m\), 
\[M_i(\alpha) := \|B_i(\alpha)\|_{P(\alpha)},\, K_{ik}(\alpha):= \|[B_k(\alpha), B_i(\alpha)]\|_{P(\alpha)},\]
and \[M_{tot}(\alpha):= \sum_{i=1}^mM_i(\alpha).\]
\begin{lemma}[$m$-mode splitting bound]
   \label{lem:explicit_mmode} 
   For any \(\tau\ge 0\) and \(\alpha \in \mathcal{A}\),
   \begin{multline*}
       \|\Phi(\tau, \alpha) - \Psi(\tau, \alpha)\|_{P(\alpha)}\\\le \tau^2 e^{\tau M_{tot}(\alpha)}\sum_{k=2}^mg(\tau M_k(\alpha))\sum_{i=1}^{k-1}K_{ik}(\alpha).
   \end{multline*}
\end{lemma}
\begin{proof}
  Fix \(\alpha\). For \(k\ge2\), 
  \begin{multline*}E_k = \Phi_k - \Psi_k = e^{\tau B_k}\Phi_{k-1} - \Psi_k \\= e^{\tau B_k}(\Phi_{k-1} -\Psi_{k-1}) + (e^{\tau B_k}\Psi_{k-1} -\Psi_k).\end{multline*}
  From submultiplicativity of norms, 
  \begin{equation}
      \|E_k\| \le \|e^{\tau B_k}\| \|E_{k-1}\| + \|e^{\tau B_k}\Psi_{k-1} -\Psi_k\|,
      \label{eq:11}
  \end{equation}
  where \(\|\cdot\|\) denotes \(\|\cdot\|_{P(\alpha)}\) for brevity. Now, 
  \begin{multline*}
      \|e^{\tau B_k}\| \le e^{\tau\|B_k\|} = e^{\tau M_k}; 
      \quad\Psi_{k-1} = e^{\tau S_{k-1}}; \\
      \Psi_k = e^{\tau(S_{k-1}+ B_k)}.
  \end{multline*}
  Lemma~\ref{lem:twofactor} with \(X = \tau B_k\) and \(Y= \tau S_{k-1}\) gives, 
  \begin{equation}
      \|e^{\tau B_k}e^{\tau S_{k-1}} - e^{\tau(S_{k-1}+B_k)}\| \le e^{(\|X\| + \|Y\|)} g(\|X\|) \|[X,Y]\|.
      \label{eq:12}
  \end{equation}
  Compute every factor, 
  \begin{multline}
      \|X\| = \|\tau B_k\| = \tau M_k; \\
      \|Y\| = \|\tau S_{k-1}\| \le \tau \sum_{i=1}^{k-1}\|B_i\| = \tau\sum_{i=1}^{k-1}M_i;\\
      \|[X,Y]\| = \tau^2\|[B_k, S_{k-1}]\| \\\le \tau^2\sum_{i=1}^{k-1}\|[B_k, B_i]\| = \tau^2\sum_{i=1}^{k-1} K_{ik}.
      \label{eq:13}
  \end{multline}
  From \eqref{eq:13},
\begin{equation}
    e^{(\|X\| + \|Y\|)} \le e^{(\tau M_k + \tau\sum_{i=1}^{k-1}M_i)}  = e^{\tau\sum_{i=1}^kM_i}.
    \label{eq:14}
\end{equation}
\eqref{eq:12}-\eqref{eq:14} give, 
\begin{equation}
    \|e^{\tau B_k}\Psi_{k-1} - \Psi_k\| \le \tau^2 e^{\tau\sum_{i=1}^kM_i}g(\tau M_k)\sum_{i=1}^{k-1}K_{ik}.
    \label{eq:15}
\end{equation}
Putting \eqref{eq:15} into \eqref{eq:11},
\begin{equation}
    \|E_k\| \le e^{\tau M_k} \|E_{k-1}\| + \tau^2 e^{\tau\sum_{i=1}^kM_i}g(\tau M_k)\sum_{i=1}^{k-1}K_{ik}.
    \label{eq:16}
\end{equation}
Define \[F_k:= e^{-\tau \sum_{j=1}^kM_j}\|E_k\|.\] Multiplying \eqref{eq:16} with \(e^{-\tau \sum_{j=1}^kM_j}\), for \(k\ge2\), we get
\begin{equation}
    F_k \le F_{k-1} + \tau^2 g(\tau M_k) \sum_{i=1}^{k-1}K_{ik}; \quad F_1 \equiv 0.
    \label{eq:17}
\end{equation}
Expanding \eqref{eq:17} up to \(k=m\) and undoing the multiplication,
\[\|E_m\| \le \tau^2 e^{\tau\sum_{j=1}^m M_j}\sum_{k=2}^mg(\tau M_k) \sum_{i=1}^{k-1}K_{ik}.\]
From the definitions of \(E_k\) and \(M_{tot}\),
   \begin{multline*}
       \|\Phi(\tau, \alpha) - \Psi(\tau, \alpha)\|_{P(\alpha)}\\\le \tau^2 e^{\tau M_{tot}(\alpha)}\sum_{k=2}^mg(\tau M_k(\alpha))\sum_{i=1}^{k-1}K_{ik}(\alpha).
   \end{multline*}
   This proves the lemma.
\end{proof}
\begin{remark}[Commutator graph and cyclic order]
\label{remark:commutator_graph}
For fixed $\alpha$, define a weighted commutator graph whose vertices are the modes
$\{1,\ldots,m\}$ and whose edge weight between modes $i$ and $k$ is
$K_{ik}(\alpha).$
The weight is symmetric in the two modes, since
$[B_k,B_i]=-[B_i,B_k]$ and hence $K_{ik}=K_{ki}$.
Lemma~\ref{lem:explicit_mmode} shows that the splitting error is
controlled by the weighted commutator edges exposed by the ordered
product
\[
    e^{\tau B_m(\alpha)}\cdots e^{\tau B_1(\alpha)} .
\]
At stage $k$, the telescoping argument compares the new factor
$e^{\tau B_k(\alpha)}$ with the averaged flow generated by the
previously accumulated modes
$B_1(\alpha),\ldots,B_{k-1}(\alpha)$. The corresponding commutator
contribution is
\[
    g(\tau M_k(\alpha))
    \sum_{i=1}^{k-1}K_{ik}(\alpha).
\]
Thus, in the dense worst case, the pairwise part of the bound contains
$\binom{m}{2}$ terms and therefore scales quadratically in $m$.
The actual size of the bound, however, is governed by the total weight
of the active commutator edges rather than by $m$ alone. Sparse or weak
commutator structure, and small duty fractions through
$B_i(\alpha)=\alpha_iA_i$, reduce the effective contribution. The bound
is also tied to the prescribed cyclic order: changing the order changes
the telescoping sequence and can alter the resulting certificate.
\end{remark}
Define the scalar bound \begin{multline*}\kappa(\tau, \alpha):= e^{-\eta \tau} + \tau^2 e^{\tau M_{tot}(\alpha)}\sum_{k=2}^mg(\tau M_k(\alpha))\sum_{i=1}^{k-1}K_{ik}(\alpha).\end{multline*}
Lemmas \ref{lem:avg-mmode} and \ref{lem:explicit_mmode} give,
\begin{equation*}
     \|\Phi(\tau, \alpha)\|_{P(\alpha)} \le \kappa(\tau, \alpha).
\end{equation*}

\begin{theorem}[$m$-mode Lyapunov contraction]
\label{thm: mmodethm}
Assume Lemma~\ref{lem:avg-mmode} holds for all \(\alpha\in\mathcal A\) with some \(\eta>0\). For \(\tau>0\), define
\[
\kappa^*(\tau):=\sup_{\alpha\in\mathcal A} \kappa(\tau,\alpha).
\]
If \(\kappa^*(\tau)<1\), then for every fixed \(\alpha \in \mathcal{A}\), the \(\tau\)-periodic switched system determined by \(\alpha\)
is exponentially stable, uniformly over \(\mathcal A\). In particular,
\begin{equation}
\|x(k\tau)\|_{P(\alpha)} \le \kappa^*(\tau)^k \|x(0)\|_{P(\alpha)},
\qquad k=0,1,2,\dots
\label{eq:18}
\end{equation}
and, if \(t=k\tau+s\) with \(s\in[0,\tau)\), then
\begin{equation}
\|x(t)\|_{P(\alpha)}
\le
e^{\,s\Gamma^\ast}\, \kappa^*(\tau)^k \|x(0)\|_{P(\alpha)},
\qquad \alpha\in\mathcal A,
\label{eq:19}
\end{equation}
where
\[
\Gamma(\alpha):=\max_{1\le i\le m}\|A_i\|_{P(\alpha)}, \quad \Gamma^\ast := \sup_{\alpha\in \mathcal{A}} \,\Gamma(\alpha) <\infty.
\]
\end{theorem}

\begin{proof}
By the preceding estimate,
\[
\|\Phi(\tau,\alpha)\|_{P(\alpha)} \le \kappa(\tau,\alpha)\le \kappa^*(\tau),
\qquad \text{for all } \alpha\in\mathcal A.
\]
Define the sampled sequence \(x_k:=x(k\tau)\). Then
\[
x_{k+1}=\Phi(\tau,\alpha)x_k,
\]
and hence
\[
\|x_{k+1}\|_{P(\alpha)}
\le
\|\Phi(\tau,\alpha)\|_{P(\alpha)}\|x_k\|_{P(\alpha)}
\le
\kappa^*(\tau)\|x_k\|_{P(\alpha)}.
\]
Iterating,
\[
\|x(k\tau)\|_{P(\alpha)}=\|x_k\|_{P(\alpha)}
\le
\kappa^*(\tau)^k\|x_0\|_{P(\alpha)},
\]
which proves \eqref{eq:18}. Now let \(t=k\tau+s\) with \(s\in[0,\tau)\). Starting from time \(k\tau\), the state evolves over an interval of total length \(s\) through a product of mode flows. Thus, by submultiplicativity of the induced norm and the estimate
\[
\|e^{rA_i}\|_{P(\alpha)}\le e^{r\|A_i\|_{P(\alpha)}},\qquad r\ge 0,
\]
we obtain
\[
\|x(t)\|_{P(\alpha)}
\le
e^{\,s\Gamma(\alpha)} \|x(k\tau)\|_{P(\alpha)}.
\]
Combining this with \eqref{eq:18} and using \(\Gamma(\alpha) \le \Gamma^\ast\) yield
\[
\|x(t)\|_{P(\alpha)}
\le
e^{\,s\Gamma^\ast}\, \kappa^*(\tau)^k \|x(0)\|_{P(\alpha)},
\]
which is \eqref{eq:19}. This completes the proof.
\end{proof}
    \begin{remark} [Stability certificate]
Define the maximal certified switching period
\[
\tau^\ast := \sup\big\{\tau>0 : \kappa^\ast(s)<1 \text{ for all } s\in(0,\tau]\big\}.
\]
Then, for every \(\tau\in(0,\tau^\ast)\), the \(\tau\)-periodic switched system indexed by \(\alpha \in \mathcal{A}\) is uniformly exponentially stable over \(\mathcal{A}\).
Thus, the two-mode stability rectangle generalizes to the product set
\[
(0,\tau^\ast)\times \mathcal A,
\]
where \(\mathcal A\subseteq\Delta^{m-1}\) is the admissible duty-fraction region.
\end{remark}

%% file: example.tex
\section {Numerical Example}

\label{example}

We illustrate the \(m\)-mode certificate on a three-mode cyclic system and its commutator-graph interpretation. Since the Lyapunov
inequalities and the maximization defining \(\kappa^\ast(\tau)\) are enforced
on a finite duty grid, the reported thresholds are sampled certificates.

\subsection{Sampled certificate computation}

We synthesize a duty-fraction-dependent quadratic metric of affine form
\begin{equation*}
  P(\alpha)=\sum_{i=1}^m \alpha_i P_i,\qquad P_i\in\mathbb S_{++}^n,  
  \label{eq:30}
\end{equation*}
over a compact admissible duty set \(\mathcal A\subseteq \Delta^{m-1}\). Let
\[
S(\alpha)=\sum_{i=1}^m \alpha_i A_i
\]
denote the averaged generator. We impose the Lyapunov constraints on a finite grid \(\mathcal{G} = \{\alpha^{(k)}\}_{k=1}^N\subset \mathcal A\). For a fixed value of \(\eta\) and a small \(\varepsilon>0\), we solve the feasibility problem
\[
\begin{aligned}
\text{find}\quad & \{P_i\}_{i=1}^m\\
\text{s.t.}\quad
& P_i \succeq \varepsilon \mathbb{I}_n,\qquad i=1,\ldots,m,\\
& S(\alpha^{(k)})^\top P(\alpha^{(k)})
  + P(\alpha^{(k)})S(\alpha^{(k)})
  + 2\eta P(\alpha^{(k)}) \preceq 0,\\
&\hfill k=1,\dots,N,\\
& \sum_{i=1}^m \operatorname{tr}(P_i)=m.
\end{aligned}
\]
The constraint \(P_i\succeq \varepsilon \mathbb{I}_n\) implies
\(P(\alpha)\succeq \varepsilon \mathbb{I}_n\) for all \(\alpha\in\Delta^{m-1}\).
The trace constraint fixes the homogeneous scaling of the Lyapunov
inequalities. Since the feasibility problem is convex for fixed
\(\eta\), the largest decay rate is computed by bisection over \(\eta\).
\begin{algorithm}[h]
 \caption{Sampled computation of the switching period}
 \begin{algorithmic}[1]
 \renewcommand{\algorithmicrequire}{\textbf{Input:}}
 \renewcommand{\algorithmicensure}{\textbf{Output:}}
 \REQUIRE matrices \(A_1, \ldots, A_m,\) prescribed cyclic order, duty grid \(\mathcal{G} \subset \mathcal{A}\), and metric parameterization \(P(\alpha)\).
 \ENSURE sampled certified switching period \(\tau^\ast_{\mathcal{G}}\).
  \STATE For each \(\alpha\in\mathcal{G}\), form the averaged generator \(S(\alpha)\).
  \STATE Use bisection over \(\eta\). For each fixed \(\eta\), solve the sampled SDP to obtain the matrices \(P_i\), and hence \(P(\alpha)\).
  \STATE For each $\alpha \in \mathcal{G}$, compute the scaled generators \(B_i(\alpha) = \alpha_iA_i\), their induced metric norms $M_i(\alpha)= \|B_i(\alpha)\|_{P(\alpha)}$, and the pairwise commutator norms $K_{ik}(\alpha) = \|[B_k(\alpha), B_i(\alpha)]\|_{P(\alpha)}$.
  \STATE For each candidate period $\tau>0$, evaluate $\kappa^\ast_{\mathcal{G}}(\tau):=$
  \[ \max_{\alpha \in \mathcal{G}}\left [ e^{-\eta \tau} + \tau^2 e^{\tau M_{tot}(\alpha)}\sum_{k=2}^m g(\tau M_k(\alpha)) \sum_{i=1}^{k-1}K_{ik}(\alpha)\right],\]
  where $M_{tot}(\alpha) = \sum_i M_i(\alpha)$.
  \RETURN $\tau^\ast_{\mathcal{G}}= \sup\{\tau>0: \kappa^\ast_{\mathcal{G}}(s)<1 \text{ for all } s \in (0, \tau]\}$.
 \end{algorithmic}
 \end{algorithm}
\subsection{Three-mode family}
Define the mode matrices $A_1$, $A_2$ and $A_3$:
\[ \begin{bmatrix}
    -1 & 4 & 0\\
    0 & -1 & 0\\
    0 & 0 & -1
\end{bmatrix},   \begin{bmatrix}
    -1 & 0 & 0\\
    -4 & -1 & 0\\
    0 & 0 & -1
\end{bmatrix},   \begin{bmatrix}
    -1 & 1.6 & 0\\
    0 & -1 & r\\
    0 & 0 & -1
\end{bmatrix}.\] 
The cyclic order is fixed as $1, 2, 3$, so the monodromy is
\[\Phi(\tau, \alpha) = e^{\tau\alpha_3A_3}e^{\tau\alpha_2A_2}e^{\tau\alpha_1A_1}.\] Direct computation gives
\[
\|[A_2, A_1]\|_2 = 16, \quad
\|[A_3, A_2]\|_2 = 6.4, \quad
\|[A_3, A_1]\|_2 = 4|r|.
\]
These are unscaled Euclidean commutator strengths; the true edge weights
entering the certificate are \(K_{ik}(\alpha)\).
The parameter \(r\) activates the edge between modes \(1\) and \(3\),
while the other two pairwise commutator strengths remain fixed. This makes the family a controlled test case for the commutator-graph interpretation in Remark~\ref{remark:commutator_graph}. We use the admissible duty set $\mathcal{A} = \{\alpha\in \Delta^2: \alpha_i \ge 0.1,\, i=1, 2, 3\}$, and the finite grid $\mathcal{G}$ contains $210$ points.
\subsection{Results}
For comparison with the sufficient certificate, define the sampled
Floquet threshold by \(\tau_F^{\mathcal G}
:=
\min_{\alpha\in\mathcal G}\tau_F(\alpha)\), where
\[
\tau_F(\alpha)
:=
\sup\{\tau>0:\rho(\Phi(s,\alpha))<1
\text{ for all }s\in(0,\tau]\}.\] 
Across the sweep \(r=0,0.2,\ldots,2.0\), the sampled Floquet
threshold was unchanged at \(\tau_F^{\mathcal G}=1.33464\), with worst sampled duty
\(\alpha\approx(0.4194,0.4516,0.1290)\). This is consistent with the block
upper-triangular form of the monodromy: the \(r\)-dependent term enters only through a feedforward block and hence does not affect the monodromy spectrum. Table~\ref{tab:three_mode} reports representative values from the sweep. The certified period decreases as the activated commutator edge \(\|[A_3,A_1]\|_2\) grows. A full sweep over \(r=0,0.2,\ldots,2.0\) showed the same qualitative
behavior: the affine metric certified all sampled duty points, and
\(\tau^\ast_{\mathcal G}\) decreased from \(0.1926\) to \(0.1469\) as the
edge weight increased.
\begin{table}[h]
\centering
\caption{Sampled certificates for
\(\mathcal G\subset\{\alpha\in\Delta^2:\alpha_i\ge 0.1\}.\)}
\label{tab:three_mode}

\begin{tabular}{c c c c c}
\hline
\(r\) & \(\|[A_3,A_1]\|_2\) (unscaled)& \(\eta^\ast\) &
\(\tau^\ast_{\mathcal G}\) &
\(\tau_F^{\mathcal G}/\tau^\ast_{\mathcal G}\) \\
\hline
0.0 & 0.0 & 0.9836 & 0.1926 & 6.93 \\
0.4 & 1.6 & 0.8749 & 0.1810 & 7.38
\\
0.8 & 3.2 & 0.8615 & 0.1729 & 7.72 \\
1.2 &  4.8 & 0.8289 & 0.1713 & 7.79
\\
1.6 &  6.4 & 0.7377 & 0.1554 & 8.59
\\
2.0 & 8.0 & 0.6753 & 0.1469 & 9.09 \\
\hline
\end{tabular}%
\end{table}\\
The example isolates a concrete source of conservatism in the
commutator-norm certificate. Increasing \(r\) strengthens the
commutator edge \(\{1,3\}\), and the certified switching period
decreases accordingly. However, this same coupling is feedforward in
the monodromy, so the sampled Floquet threshold is unchanged. Thus the
certificate penalizes a norm-visible commutator interaction that is
spectrally invisible in this family. The certificate is therefore most
useful as a computable uniform sufficient certificate over duty sets and
prescribed cyclic orders, especially when a single certificate is
desired across many schedules, rather than as an exact Floquet
characterization for a fixed duty vector.